\documentclass[12pt]{amsart}
\usepackage{latexsym,color,amssymb,amsmath,amsthm,graphicx,listings,comment}
\usepackage[section]{placeins}

\newtheorem{propo}{Proposition} 
\let\myparagraph\subsection
\definecolor{bbyellow}{rgb}{1.0,1.0,0.8}

\title{On Primes, Graphs and Cohomology}
\author{Oliver Knill}
\date{August 20, 2016}
\address{Department of Mathematics \\ Harvard University \\ Cambridge, MA, 02138, USA }

\begin{document}
\maketitle

\begin{abstract}
The counting function on the natural numbers defines a discrete Morse-Smale complex with a
cohomology for which topological quantities like Morse indices, Betti numbers
or counting functions for critical points of Morse index are explicitly given in number theoretical
terms. The Euler characteristic of the Morse filtration is related to the Mertens function, 
the Poincar\'e-Hopf indices at critical points correspond to the values of the Moebius function.
The Morse inequalities link number theoretical quantities like the prime counting functions
relevant for the distribution of primes with cohomological properties of the graphs. 
The just given picture is a special case of a discrete Morse cohomology equivalent to simplicial
cohomology. The special example considered here is a case where
the graph is the Barycentric refinement of a finite simple graph. 
\end{abstract}

\section{Summary} 

\myparagraph{}
For an integer $n \geq 2$, let $G(n)$ be the graph with vertex set $V_n=\{ k \; | \; 2 \leq k \leq n, k \; {\rm square} \; {\rm free} \;\}$ 
and edges consisting of unordered pairs $(a,b)$ in $V$, where either $a$ divides $b$ or $b$ divides $a$. We
see this sequence of graphs as a Morse filtration $G(n) = \{ f \leq n \}$, where
$f$ is the Morse function $f(n)=n$ on the Barycentric refinement $G$ of the 
complete graph $P$ on the spectrum of the ring ${\mathbb Z}$. 
The Mertens function $M(n)$ relates by $\chi(G(n))=1-M(n)$ to the Euler characteristic 
$\chi(G(n))$.  By Euler-Poincar\'e, this allows to express $M(n)$ cohomologically 
as the sum $\sum_{k=1}^n (-1)^k b_k = 1-M(n)$ through Betti numbers $b_k$
and interpret the values $-\mu(k)$ of the M\"obius function as Poincar\'e-Hopf 
indices $i_f(x)=1-\chi(S^-_f(x))$ of the counting function $f(x)=x$. The additional $1$ in the Mertens-Euler
relationship appears, because the integer $1$ is not included in the set of critical points. 
The summation of the M\"obius values illustrates a case for the Poincar\'e-Hopf 
theorem $\sum_x i_f(x)=\chi(G(n))$ \cite{poincarehopf}.
The function $f$ is Morse in the sense that in its Morse
filtration $G(x) = \{ f \leq x \; \}$ the unit sphere $S(x)$ of any integer $x$, 
which was added last, is a graph theoretical sphere. The Morse index is then 
$m(x)={\rm dim}(S(x))+1$ and the strong Morse inequalities become elementary in this
case. 

\myparagraph{}
This just summarized story shows that {\bf ``counting" is a Morse theoretical 
process}, during which more and more `handles" in the form of topological 
balls are added, building an increasingly complex topological structure. 
When working on the graph $N$ of all natural numbers larger than $1$
with two numbers connected if one divides the other, substantial topological changes happen 
at square free integers while harmless homotopy deformations take place at the other points. 
In the homotopically equivalent graph $G$, where only square free vertices are considered,  
every point is a critical point. The dimension $m(x)$ of the handle 
formed by adding $x$ defines a Morse index $m(x) \in \{0,1,2 \dots \;\}$ satisfying 
$i_f(x)=(-1)^{m(x)}$ and allowing to define $c_m(x)$, 
the {\bf number of critical points} with index $m$ on $G(x)$.
The quantity $c_0(x) = \pi(x)+\pi(x/2)$ for example counts the zero-index critical point
instances, at which isolated vertices in the form of primes appear and disappear.
It relates the prime counting function $\pi(x)$ with a connectedness property
of the topological space. 
As in the continuum, the strong Morse inequalities relate the Betti numbers $b_k$ with
the counting functions $c_k$ of critical points. The zero'th Betti number 
is $b_0(G(x)) = 1+\pi(x)-\pi(x/2) \leq c_0(x)$. For $k>0$ we have
$b_k(G(x)) = \pi(k,x)-\pi(k,x/2) \leq c_k(x)=\pi(k,x) + \pi(k,x/2)$, where $\pi(k,x)$ counts the
number of prime $k$ tuples whose product is smaller or equal than $x$. The weak
Morse inequalities $b_k \leq c_k$ are trivial, the strong inequalities follow 
from the exponential decay of $c_k(x)-b_k(x)=2 \pi(k,x/2)$ in $k$. 

\myparagraph{}
If $n$ is a {\bf Kummer number} meaning that if $n+1$ is the product of the first $d+1$ rational primes,
then the graph $G(n)$ has an involutive {\bf time reversal symmetry} $k \to (n+1)/k$ which leads 
to {\bf Poincar\'e duality} $b_k=b_{d-k}$ for the Betti numbers of the connected component of the graph
$G(n+1)$. The connected component of $G(n+1)$ (disregarding the isolated primes) is the Barycentric refinement 
of $K_{d+1}$, the $d$-dimensional simplex formed on the first $d+1$ primes. 
We get a discrete version of a Morse-Smale system in which every 
vertex is a discrete analogue of a hyperbolic critical point whose stable and unstable 
manifolds behave as if they would intersect transversely. 
The Morse cohomology on this Morse-Witten complex is equivalent to the cohomology of 
the graph $P(n)$ which by Kuenneth is equivalent to the cohomology of the Barycentric 
refinement $G(n)$. There are other connections to number theory. 
The identity $\sum_{1 \neq d|n} \mu(d)=-1$ for example is a manifestation of
the fact that a ball has Euler characteristic $1$. 

\myparagraph{}
This arithmetic example of the natural numbers is a prototype. In full generality,
for any suitably defined Morse-Smale function on a finite simple graph, its Morse cohomology is
equivalent to simplicial cohomology. While this is not surprising given that the result holds
in the continuum, it is crucial to have the right definitions of stable and unstable manifolds
at critical points and a correct notion of Morse-Smale in the discrete. It turns out that these
definitions can be done for general finite simple graphs regularized by their Barycentric refinement.
The set of Morse functions might be pretty limited but if the graph is a $d$-graph meaning that
every unit sphere is a (purely graph theoretically defined) $(d-1)$-sphere called {\bf Evako sphere},
then the notion of Morse-Smale pretty much looks like the one in the continuum: every critical 
point must have the property that $S^-_f(x),S^-_{-f}(x)$ are spheres and that the 
corresponding stable and unstable manifolds for two different critical points $x,y$ of 
index $m,m-1$ intersect transversely or are disjoint. This allows
to define an intersection number $n(x,y)$ of two critical points through stable and unstable
manifolds and an exterior derivative just like in the continuum. The discrete
setup also illustrates why Floer's generalization to non-orientable and infinite dimensional setups
works. We don't need orientability of the graph in the discrete. 

\myparagraph{}
Given a finite simple graph $G=(V,E)$. 
A function $f:V \to \mathbb{R}$ is {\bf Morse} if it is locally injective and each subgraph of 
the unit sphere $S(x)$ generated by 
$S^-_f(x) = \{y \in S(x) \; | \; f(y)<f(x) \; \}$ is either contractible or a sphere. It is {\bf Morse-Smale}, 
if at every vertex $x \in V$, discrete stable and unstable manifolds $W^{\pm}_f(x)$ exist which pairwise either
do not intersect or intersect transversely in the sense that $W^-_f(y) \cap S(x) \subset S^-_f(x)$, allowing to 
define an intersection number $n(x,y)$ between different critical points: it is $1$ if the orientations
agree and $-1$ if the orientations don't agree. The Morse complex is then
defined exactly as in the continuum: the vector space $\Omega(G,f)$ of functions on critical points 
of $f$ is filtrated by Morse index: $\Omega_m(G,f)$ is the set of functions on critical points of 
Morse index $m$ and admits an exterior derivative $d: \Omega_m \to \Omega_{m+1}$ given by 
$dg(x) = \sum_{y} n(x,y) g(y)$. Morse cohomology is useful, as it is easier to compute as simplicial 
cohomology. In the discrete, it is also equivalent to \v{C}ech cohomology which is the simplicial 
cohomology of a nerve graph. The Morse function does the job of triangulating the graph.

\myparagraph{}
We expect Morse functions in the discrete to be abundant:
any injective function $f$ on the vertex set of a geometric graph $G$ should have
a natural extension to its Barycentric refinement $G_1$, where it is Morse in the sense
that adding a new vertex $x$ in the Morse filtration $G(x) = \{ f \leq x \; \}$ is either a 
homotopy deformation or then adds a handle in form of a graph theoretical $k$-ball $B^-_f(x)$, meaning that the 
level surface $\{ f=c \; \}$ for $c \notin {\rm Ran}(f)$ is always a graph theoretical 
sphere near the critical point $x$. The case of $f(x)={\rm dim}(x)$ and especially the 
arithmetic example introduced here illustrates that the graph does not need to have a discrete manifold structure 
in order for having a discrete Morse theory. In this note, we look only at the special case 
$f(x)={\rm dim}(x)$ on the Barycentric refinement of a general finite simple graph $G$. 
It will follow almost from the definitions that in this case, the Morse function 
$f(x) = {\rm dim}(x)$ on the Barycentric refinement $G_1$ of $G$ has a Morse cohomology which is identical 
to the simplicial cohomology of the original graph $G$. Because simplicial cohomology of $G$ is isomorphic
to simplicial cohomology of $G_1$ of a general statement that Morse cohomology is equivalent to
simplicial cohomology. In the continuum this is a well developed theory 
by Morse, Thom, Milnor, Smale or Floer. For a general theory for finite simple graphs, we need to 
restrict the Morse functions to a discrete class of Morse-Smale functions and verify that the 
Morse cohomology obtained is equivalent to simplicial cohomology on the graph. 
Since as just pointed out, there is always at least {\bf one} Morse function, we 
only have still to establish conditions under which deforming $f$ keeps the exterior derivative
defined and the corresponding cohomologies unchanged. This will require only the study of local
bifurcations of critical points. We plan to work on that elsewhere. 

\myparagraph{}
Denote by $B^-_f(x)$ the graph generated by the set of vertices in $S(x)$ for which 
$f(y) \leq f(x)$ and let $S^-_f(x)$ the graph generated by vertices in $S(x)$ for which $f(y) < f(x)$. 
If $S^-(x)$ is a discrete sphere so that $\chi(S^-(x))$ is either $0$ or $2$, 
the Morse index $m(x)={\rm dim}(B_-f(x))$ at $x$ is defined together by the Poincar\'e-Hopf index
as $i_f(x)=(-1)^{m(x)} = 1 - \chi(S^-_f(x))$.
It measures the change of Euler characteristic when adding the handle $B_f^-(x)$.
The discrete gradient flow defines a Morse-Smale complex whose Thom-Smale-Milnor cohomology is equivalent
to simplicial cohomology. In the case of a Barycentric refinement, it is equivalent to the original 
cohomology of $G$. If $c_k$ is the number of critical points of Morse index $k$, and $c(t)=\sum_k c_k t^k$ and
$p(t)=\sum_k b_k t^k$ are the Morse and Poincar\'e polynomials, then the strong Morse inequalities hold:
$c(t)-b(t)=(1+t)r(t)$, where $r(t)$ has positive coefficients. In our case this can be shown directly.
In general, an adaptation of the semi-classical analysis using Witten deformation \cite{Witten1982,Cycon} is needed. 
The {\bf Witten deformation} argument actually shows more: the heat flow of any of the deformed 
Laplacian $L_s$ defined by $f$ not only defines the Hurewicz homomorphism from the fundamental 
groups to the cohomology groups $H^k(G)$, it can be used to match a harmonic form representing a 
cohomology class with the sphere $S^-_f(x)$ defined
at the critical point $x$ of $f$ as in the limit $s \to \infty$ the support of any harmonic $k$-form
is on the union of $k$-spheres associated to critical points.

\section{Graphs}

\myparagraph{}
The {\bf natural numbers} $\mathbb{N}=\{1,2,3 \dots \}$ define a simple graph called the
{\bf integer graph} $N$. Its vertex set is $\mathbb{N}$. The edges are the pairs of
distinct integers for which one divides the other. This graph $N$ is {\bf simple}
as we allow no self loops, nor multiple connections.
Since $N$ is the unit ball of the integer $1$ and the unit $1$ is not of interest,
we better restrict to the unit sphere $S(1)$ of $1$. This graph $S(1)$ in turn is related 
to a more naturally defined subgraph which we call the {\bf prime graph} $G$. The vertices
of the prime graph $G$ are made of all square free integers in $N$. The graph $G$ can be identified with the 
{\bf Barycentric refinement} of the {\bf spectrum} $P$ of the commutative ring $\mathbb{Z}$, 
as the vertex set of $G$ is the set of square free integers different from $1$. The spectrum $P$, the set of
rational primes is a graph where all vertices are connected to each other. 
The graph $G$ then generates a sequence of graphs $G(n)$ generated by the vertex set 
$\{2,3, \dots, n\} \cap \mu^{-1}(\-1,1\})$ of all square free integers smaller or equal to $n$. 
Similar graphs $G_R$ could be constructed from any countable commutative ring $R$: just
take the Barycentric refinement of the spectrum of $R$ in which all vertices are connected.
On the finite ring $\mathbb{Z}/(n \mathbb{Z})$ for example, the graph would be the 
Barycentric refinement of the complete graph on the set of all prime factors of $n$. 

\myparagraph{}
We start with some graph theoretical notations.
Given any finite simple graph $G=(V,E)$ with {\bf vertex set} $V$ and edge
set $E$, a complete subgraph $K_{k+1}$ of $G$ with $k+1$ 
vertices is called a {\bf $k$-simplex}. The {\bf Euler characteristic} of $G$ is 
defined as $\sum_{k=0} (-1)^k v_k(G)$, where $v_k(G)$ is the number of 
$k$-simplices in $G$. Given a subset $W$ of the vertex set $V$, the graph {\bf generated} by $W$
is the graph $(W,F)$, where $F$ is the subset of all edges $E$ for which all 
attached vertices are in $W$. The {\bf unit sphere} $S(x)$ of a vertex $x$ in $G$ is the
graph generated by all the vertices directly connected to $x$. 
In the graph $N$ for example, the entire graph ring is the unit
ball of $1$ with Euler characteristic $1$ and look at the
unit sphere $S(1)$. In the graph $G(30)$ for example, the unit sphere of $v=30$ is
the cyclic sub graph generated by the vertex set $\{2,6,3,15,5,10 \;\}$. In $G(105)$, the unit
sphere of the latest added point $v=105$ is a stellated cube, which is a $2$-sphere.

\myparagraph{}
Following Evako, graph theoretical spheres are defined by induction \cite{KnillJordan}
first of all, the empty graph $G=(\emptyset,\emptyset)$ 
is a $(-1)$-sphere. Inductively, for $k \geq 0$, a {\bf $k$-graph} is a finite simple graph for which 
every unit sphere is a $(k-1)$-sphere. A $k$-sphere is a $k$-graph which after removing a vertex 
becomes contractible. A finite simple graph $G=(V,E)$ is {\bf contractible}, 
if there exists a vertex $v$ such that both the unit sphere $S(v)$ as well as the graph
generated by $V \setminus \{v\}$ are contractible. Two graphs are {\bf homotopic}
if one can get from one to the other by {\bf local deformation steps}. 
A deformation step either removes a vertex
$v$ with contractible $S(v)$ or adds a new vertex $v$ connected to a 
contractible subgraph $(W,F)$ of $(V,E)$. 
All these notions are combinatorially defined without referring to any Euclidean 
{\bf geometric realization}. The later would lead to classical topological notions 
but they are not needed. 

\myparagraph{}
Given two finite simple graphs $H,G$, the {\bf Cartesian product} $G \times H$ is defined 
as the graph whose vertices are the pairs $(x,y)$, where $x$ is a simplex in $G$ and 
$y$ is a simplex in $H$ and whose edge set consist of all $((a,b),(c,d))$, where 
either $(a,b)$ is contained in $(c,d)$ or $(c,d)$ is contained in $(a,b)$. 
This graph product has been ported from simplicial complex constructions to graph theory \cite{KnillKuenneth}
and has nice properties like that $\chi(H \times G) = \chi(H) \chi(G)$,
where $\chi(G) = \sum_{x \subset G} (-1)^{{\rm dim}(x)}$ is the Euler characteristic
of $G$ summing over all simplices $x$ in $G$. A special product is $G \times K_1$ which
is the {\bf Barycentric refinement} of $G$. Its vertices are the simplices of $G$
and two simplices are connected, if one is contained in the other. 
The product has all the properties we know from the continuum. It satisfies 
the K\"unneth formula for cohomology \cite{KnillKuenneth}. It has dimension 
${\rm dim}(x) = 1+\sum_{x \in V} {\rm dim}(S(x))/|V|$
which satisfies in general the inequality 
${\rm dim}(G \times H) \geq {\rm dim}(G) + {\rm dim}(H)$
which is familiar from Hausdorff dimension in metric spaces even so the discrete dimension
just defined has no formal relations to Hausdorff dimension in the continuum. 

\section{Morse theory}

\myparagraph{}
Lets look at an injective real-valued function $f:V \to \mathbb{R}$ on a countable simple graph $G=(V,E)$
so that we have a {\bf filtration} $G(n) = \{f \leq n \; \} \subset G(n+1)$ of finite simple subgraphs of $G$.
A locally injective function $f$ on the vertex set of a graph $G$ is called a {\bf Morse function}
if for every vertex $x \in V$, the graph $S^-_f(x)$ generated by
the set of vertices $y$ for which $f(y)<f(x)$ is a {\bf $m$-sphere} for some $m=m(x) \geq -1$. 
We call $S^-_f(x)$ the {\bf stable sphere} of $f$ at $x$. 
Adding a new prime $p$ for example to the graph $G(p-1)$ produces a critical point for which the 
unit sphere is the empty graph, and the empty graph is by definition the $(-1)$-sphere. 
When adding a vertex $x=pq$ with two primes $p,q$ which have not yet been
connected anywhere yet, then $x$ is a critical point for which a $0$-sphere has been added. 
By definition, a $0$-sphere is a $0$-graph (meaning that all unit spheres are $-1$-spheres
and that removing one vertex produces a contractible graph, so that a $0$ sphere consists of two
isolated discrete points). 
The Morse index of that added point is then $1$ because the added disc $B(x)$ is a 
$1$-ball which has a $0$-sphere as its boundary. In the prime graph $G$, this happens for $x=6$ for 
the first time. If $x=pq$ connects to two primes which are already connected like
for $x=15$, where $3-6-2-10-5$ already links $3$ and $5$, we close a loop. 
We see that adding a new vertex $p_1 p_2 \dots, p_k$ either destroys a $(k-2)$-
sphere or creates a $(k-1)$-sphere. 

\myparagraph{}
The {\bf stable sphere} $S^-_f(x)$ of $f$ at a vertex $x$ can be seen as
the intersection of the {\bf stable part} $\{ y \in G \; | \; f(y)<f(x) \; \}$ 
of the gradient flow of $f$ with the unit sphere in the full graph. 
Similarly, $S^-_f(x)$ generated by the vertices in the unit 
sphere of $x$, where $f(y)>f(x)$ is the {\bf unstable sphere}. 
In the case of the prime graph, the stable part of $x$ 
is just the set of vertices in the graph consisting of 
numbers dividing the vertex. Remarkably, the stable sphere 
is always a $m$-sphere for some $m$ allowing to define the
Morse index. Given two vertices $x,y$ the {\bf heteroclinic
connection} between $x$ and $y$ is the intersection of the stable
and unstable manifolds of $x$ and $y$. The heteroclinic
connection between $x=6$ and $y=210$ for example consists of
the vertices $\{6,30,42,210 \; \}$. 

\myparagraph{}
Since every vertex is a critical point, there is no discrete
gradient flow on the graph itself, at least when trying to define it
as in the traditional sense. The reason is that
every point would be a stationary point.
We still can think of the stable parts and unstable parts of the unit spheres as
part of ``stable and unstable manifolds". It would be possible to get
a classical gradient flow by doing a geometric realization of the 
simplicial complex and then define a gradient
flow on that manifolds. The easiest would
be to embed $G(x)$ in $R^{n}$ if $x$ is a product of $n$ primes
and map every point $y = p_0 \dots p_k$
to a point $x$ for which $x_{p_j}=1$ for every $0 \leq j \leq k$ and
$0$ else. Then define a suitable function $f(x)$ which embeds the
discrete geometry into the continuum. Doing so would be
inelegant as finite combinatorial problems should be dealt using finite sets. 

\myparagraph{}
A first observation which led us to write this note down is the realization
that the counting function $f(x)=x$ is a {\bf Morse function}
on the prime graph $G$: its Morse filtration is homotopic to the 
Morse filtration of the integer graph $S(1) \subset N$ and the Euler characteristic is 
related the {\bf Mertens function} $M(x) = \sum_{k=1}^n \mu(k)$ by the formula
$\chi(G(x))= 1-M(x)$. The {\bf Poincar\'e-Hopf index} $i_f(x) = 1-\chi(S^-_f(x))$, 
is the {\bf M\"obius function} $\mu(x)$ which is $1$ if $x$ is a product of an even
number of distinct primes and $-1$ if $x$ is a product of an odd number of distinct
primes and $0$ if there exists a square factor different from $1$ in $x$. At points, where
$\mu(x)=0$, the graph extension $G(x-1) \to G(x)$ is a homotopy deformation step, during
which no interesting topology changes happens. When restricting to the graph $G(n)$ 
whose vertices are the square-free integers $\leq n$, every vertex in the graph 
is a critical point. The formula $\chi(G(x)) = 1-M(x)$ can be seen as the 
Poincar\'e-Hopf formula, where the indices are $-\mu(x)$ and the Mertens function
includes the index $1$ of the point $1$, which we do not consider. Note that the 
integer graph $N = B(1)$ for which the vertex set is $\mathbb{N}$ 
has Euler characteristic $\chi(N)=1$ as it is the unit ball of a vertex. 

\myparagraph{}
At critical points $x$ of $f$, where $x$ gets
attached to a sphere $S^-_f(x)$ with Euler characteristic
$0$ or $2$, the {\bf Poincar\'e-Hopf index} is $i_f(x) = \pm 1$. 
If $B_f^-(x)$ is the graph generated by $x$ and $S^-(x)$, then 
the dimension $m=m(x)$ of the stable ball $B_f^-(x)$
is called the {\bf Morse index} of $x$. 
For a Morse function $f$ on a finite simple graph, denote by $c_m(f,G)$,
the number of critical points with Morse index $m$. In the case $G(x)$, we simply
write $c_m(x)$. The Morse counting functions $c_m(f,G)$ 
are of interest: for the Morse function $f(x)=x$ on the integer graph $G$ or 
prime graph $P$, we have $c_0(x) = \pi(x)+\pi(x/2)$ with {\bf prime counting function},
as every prime $p\leq x$ is a critical point with Morse index $0$ (the handle $\{x\}$
with $-1$ dimensional unit sphere $\emptyset$ has been added) 
and every number $2p \leq x$ is a critical point with Morse index $1$
(the $1$-ball $\{2,p,2p\}$ with 0-dimensional unit sphere $\{2,p\}$ was added).
As the prime counting function, also the functions $c_k(x)$ for larger $k$ could
be interesting. To make them accessible by Morse theory, we need cohomology. 

\section{Cohomology}

\myparagraph{}
Given a finite simple graph like $G(x)$, we can assign an orientation to its
maximal simplices. This does not need to be an orientation of the entire 
graph as we don't require the orientations to be compatible on 
intersections of two simplices. The choice of the orientations will not matter. It corresponds
to a choice of basis in a vector space. Cohomology is by Hodge a spectral notion which does
not depend on the choice of this basis: the groups are the kernels of Laplacians which are orientation
independent. Let $\Omega^k(G)$ denote the set of functions on $k$-simplices such that
changing the orientation of a simplex $x$ changes the sign of $f$. These are discrete
$k$-forms. Define the exterior derivative as $df(x) = f(dx)$, where $dx$ is the boundary chain of $x$.
Since the boundary operation $d$ satisfies $d^2 x=0$, the exterior derivative operation satisfies
$d^2 f=0$. The $k$'th {\bf cohomology group} $H^k(G)$ is defined as 
${\bf ker}(d|\Omega^k)/{\bf im}(d|\Omega^{k-1})$. The Hodge Laplacian $L=(d+d^*)^2$
is a $v \times v$ matrix with $v=\sum_{i=0}^{\infty} v_i$. It splits into blocks $L_k$. 
The dimension $b_k(G)$ of the vector space $H^k(G)$ is called the $k$'th {\bf Betti number}
of $G$. By Hodge theory, $b_k(G)$ is the nullity of the matrix $L_k$. 
For example, $b_0(G)$ is the number of connected components of $G$. 

\myparagraph{}
By the {\bf Euler-Poincar\'e formula}, we can now write 
$\chi(G(x)) = \sum_{k=0}^{\infty} (-1)^k b_k(x)$ with {\bf Betti numbers} $b_k(G)$.
Since we have a finite graph, this sum is finite. 
We see that the Mertens function has a cohomological interpretation. 
and that the {\bf Morse index} $m(x)$
which gives the dimension of the attached or removed handle
satisfies $-\mu(x) = i_f(x) = (-1)^{m)} = (b_m(x)-b_m(x-1))$. 
More number theoretical functions which relate to cohomology through 
Morse theory are obtained by counting critical points. 
We look at them next. 

\myparagraph{}
The number $c_m(x)$ of critical points for which $S^-_f(x)$ has dimension $m-1$ 
satisfies the {\bf Morse inequalities} $(-1)^p \sum_{k \leq p} (-1)^k [c_k-b_k] \geq 0$ or
$(-1)^p \sum_{k > p} (-1)^k [c_k-b_k] \leq 0$. The two just given statements are
equivalent since the {\bf Euler-Poincar\'e identity} reads $\sum_{k=0} (-1)^k [c_k-b_k] =0$.
Witten's idea using the deformed derivative $e^{-s f} d e^{s f}$ would
give even more information using semi-classical analysis.
The Morse counting functions $c_k(G)$ are of some
interest, as $c_0(x)$ grows like the {\bf prime counting function} $\pi(x)$. 
The function $c_1(x)$ related to closed loops 
grows at least like the {\bf prime triple counting function} $\pi_3(x)$
counting the number of triples $p,q,r$ of primes with $pqr \leq x$. 
More precisely, $c_1(x)  \geq  |\{ (p<q<r) \; | \; qr \leq x, pqr>x \; \}|$
$+|\{ (p<q<r) \; | \; pqr \leq x \; \}|$ which counts instances where circles
$p,pq,q,qr,r,rp$ were ``born" plus the number of discs $p,pq,q,qr,r,rp,pqr$
which are burial grounds, where circles have "died". 
The {\bf weak Morse inequalities} $b_k \leq c_k$ already lead to 
$\sum_{k=0}^n b_k \leq \pi^2 n/6$.

\myparagraph{}

Counting is a dramatic process which tells the story of the
life and death of spheres. While we can only look at the events,
where a square free number is added, we can also look at the full story. 
For $n=1$, space is the empty graph. At $2,3$, new $0$-balls are added.
At $n=4$ nothing interesting happens, as the longly $2$ forming a graph $K_1$
gets a homotopy deformation to a $K_2$.
At $n=5$ another $0$-sphere is added. At $n=6$, we see the creation of 
a $1$-ball and a destruction of the $0$-sphere $\{2,3\}$. 
At $n=7$ an other $1$-ball is added. At times $n=8$ and $n=9$, nothing spectacular 
takes place as we observe just homotopy deformations. 
At time $n=10$, a $1$-ball gets glued to the sphere $\{2,5\}$.
The first 1-sphere is born at $n=15$. This circle dies at $n=30$. The first 2-sphere 
sees the light at $n=105=3*5*7$ and dies at $n=210=2*3*5*7$. 
The unit sphere of a new square-free vertex $n$ is always a sphere $S_f(x)$ 
and the change of Euler characteristic is the index $i_f(x) = 1-\chi(S_f(x))$. When
adding $x$, the ball $B_f(x)$ is the handle which gets attached to the sphere $S_f(x)$. 
It either destroys the sphere $S_f(x)$ and adds a new sphere.

\myparagraph{}
The Betti numbers $b_k(G(x))$ reflect on essential parts of the geometry and so report
on more interesting feature of the topology of the graph $G(x)$.
Classically, Betti numbers were estimated by other means, especially
in terms of curvature and diameter. While we don't pursue this here, the prime graphs
give us a source of examples of graphs, where we can compute all 
cohomologies without actually computing the vector spaces from exterior derivative explicitly 
and where arbitrary high Betti numbers matter. One can speculate that some
work on curvature estimates could eventually shed some light on the Riemann hypothesis.
But certainly this does not happen on the elementary level we treat the topic here. 
We look at things the opposite way: we have 
a source of examples of graphs, where we can compute arbitrary high cohomology groups 
by other means. We have computed the cohomology groups up to $n=250$ the hard way 
explicitly using Hodge by computing the kernels of the Laplacians. 
It could be possible that the actual representatives of the cohomology groups, the 
{\bf harmonic forms}, vectors in the kernel of $L_k$ are of number theoretical
interest. Also still unexplored spectral properties 
of the form-Laplacians $L_k$ could be of number theoretical interest. 

\section{More remarks}

\myparagraph{}
It is more natural to use the {\bf ring} $\mathbb{Z}$ of integers rather than the semi-ring 
$\mathbb{N}$ of natural numbers. The {\bf spectrum} $P$ of $\mathbb{Z}$ is the set of {\bf prime ideals} 
in $\mathbb{Z}$ which naturally can be identified with the set of {\bf prime numbers}. 
An {\bf abstract simplicial complex} on $P$ is given by the collection $K$ of all 
finite non-empty subsets of $P$. It is the Whitney complex of the complete graph defined on $P$. 
The Barycentric refinement of the complex $K$ is the Whitney complex of the graph for
which the vertices are the square free positive integers different from $1$
and where two integers are connected, if one is a divisor of the other. As described in the
introduction, we call this the {\bf prime graph}. 

\myparagraph{}
Here is a bit more mundane approach which motivates to look at square free
integers: consider the finite set $P_n$ of primes in $I_n = \{ 2, \dots, n \;\}$.
The set of all subsets of $P_n$ for which the product is in $I_n$ is 
still a simplicial complex even so it is {\bf not a Whitney complex of a graph}. 
However, its {\bf Barycentric refinement} $K$ is the Whitney 
complex of a finite simple graph $G(n)=(V(n),E(n))$ which we call 
the {\bf Morse filtration} of the {\bf prime graph}:
the vertices $V(n)$ are the set of {\bf square free numbers} in $I_n$ 
and two vertices are connected if one is a divisor of the other. 
the {\bf prime numbers} are the atomic points and the division structure equips
$V(n)$ with a partial order structure. 
The integer $35$ for example contributes to the sets $\{3\},\{5\},\{3,5\}$ to $K$
It can be seen as a "line segment" as the intersecting ideal $[5\times 7]$ 
which connects the ideals $[5]$ and $[7]$. More generally,  
any square-free integer $x=p_1 p_2 \cdots p_k$ defines  a $(k-1)$-dimensional 
simplex or faces in $K$. It becomes a point in the prime graph $P$ containing
the finite graphs $G(n)$. The M\"obius function $\mu(x)$ which originally assigns
to a simplex $x$ of dimension $k$ the {\bf valuation} $\mu(x) = (-1)^k$ is now a 
{\bf function} on the vertices of $G(n)$ and it is the Poincar\'e-Hopf index $i_f(x)$ of the
counting function $f$.

\myparagraph{}
The Euler characteristic $\chi(G(n))$ of the prime graph is now equal to $1-M(n)$ as 
$1$ does not belong to the spectrum of $Z$. Geometrically speaking, $G(n)$ is a subgraph of the
{\bf unit sphere} of $1$ in $G(n)$. The entire integer graph $G(n)$ is a homotopy extension of the
{\bf unit ball} of $1$. Since the $k$-simplices in the unit ball of $1$ belong to
$(k-1)$-simplices in the unit sphere we get a sign change in $\chi$ and a shift in 
the dimension of cohomology for $k \geq 1$. 
Because nothing interesting happens at points $x$ which are a multiple of a square prime,
the Barycentric refinement $G(n)$ of the spectral simplicial 
complex is equivalent to the integer graph $G(n)$. Actually, including a vertex $n$ which 
contains a square larger than $1$ would produce a {\bf homotopy deformation} of the graph 
and not change the Euler characteristic. For example, removing $9$ from $G(12)$ does not
change its topology as the unit sphere of $9$ is the contractible set $\{3\}$.
We stick with the prime graphs $G(n)$ containing $Q(x)$ vertices 
rather the graphs $G(n)$ with vertex set $\{1, \dots,n \}$. 
Since the Basel problem renders $Q(x)$ proportional to $x$
with proportionality factor $\pi^2/6$, not much is gained from this homotopy reduction,
but it helps for visualization purposes to discard the homotopically irrelevant parts.
Figure~(\ref{figure1}) visually illustrates this.

\myparagraph{}
Euclid saw ``points" as `that which has no part", line segments as
a connection between two points and ``triangles" as a geometric objects defined by three points. 
The graph $G(n)$ is a geometric object containing points, line segments, triangles and
higher dimensional simplices, where the geometrically defined quantity $\chi(G(n))$ is 
of number theoretical interest. The Mertens function is a determinant of the Redheffer 
matrix and now also cohomologically expressible as kernels of matrices. 
The Moebius function $\mu(x)$ is the Poincar\'e Hopf index of
the scalar Morse function $f(x)=x$. If the $k$'th Betti number changes, 
we add a $k$-dimensional handle. For each $n$, the trivial weak Morse inequalities 
$b_p \leq c_p$ generalize to the statement that $c(t)-b(t)=(1+t)R(t)$ where $R(t)=\sum_k r_k t^k$
satisfies $r_k \geq 0$. The {\bf weak Morse inequalities} compare the coefficients of $c-b$
the {\bf strong Morse inequalities} are obtained by comparing the $t^p$ coefficients of 
$\sum_k r_k t^k = (\sum c_k t^k - \sum b_k t^k)(1-t+t^2 \dots)$.

\myparagraph{}
The graph $G(n)$ can contain "dust" in the form of isolated vertices. 
These zero-dimensional part consists of primes in $(n/2,n]$ not yet attached to the main graph.
Then there are "hairs", which are line segments $p,pq,q$ of primes $p,q$, where
$q$ is not connected to any other prime. 
Then there are basic "loops" $p,pq,q,qr,r,rp$ with triples of 
primes, where none of the composite numbers are connected to anything else. 
Two dimensional "disks" are obtained by triples of primes like that 
but with an additional central vertex $rpq$. Increasing $n$
reaching such a vertex is an example where we add a two dimensional handle. 
The first embedded 2-sphere only appears for $n=105=3*5*7$. Only at 
$n=1155=3*5*7*11$, the first $3$-sphere is present in $G(n)$. 

\myparagraph{}
If $n$ is the {\bf Kummer number} $n=2 \cdot 3 \cdot \dots \cdot p-1$, which is one less
than the product of the first $k$ primes, then the graph $G(n)$ has counting as a Morse-Smale
system. The reason is that we have then a {\bf time reversal involution} $k \to (n+1)/k$ 
on $G(n)$. While in general, the graphs $G(n)$ have a trivial automorphism group, in the case of the
Kummer number, we have an involution. Now, we not only have the stable manifold of a vertex
but also an unstable one. Every vertex $v$ is a {\bf hyperbolic critical point}. 
The stable sphere $S^-_f(x)$ can be seen as the analogue of defining the stable manifold. Due
to time reversal, we also have an unstable sphere $S^+_f(x)$ as part of the unstable manifold. 
Now given two vertices $x,y$ in $G(n)$, since both are critical points, we can look at the
{\bf connecting manifold} which can be empty. For example, in $P(209)$, we can connect the
two vertices $2$ and and $70$ via the homoclinic connection $2,10,70$. We have now
what one calls a {\bf Morse-Smale system}.

\section{Illustrations}


\begin{figure}[!htpb]
\scalebox{0.1}{\includegraphics{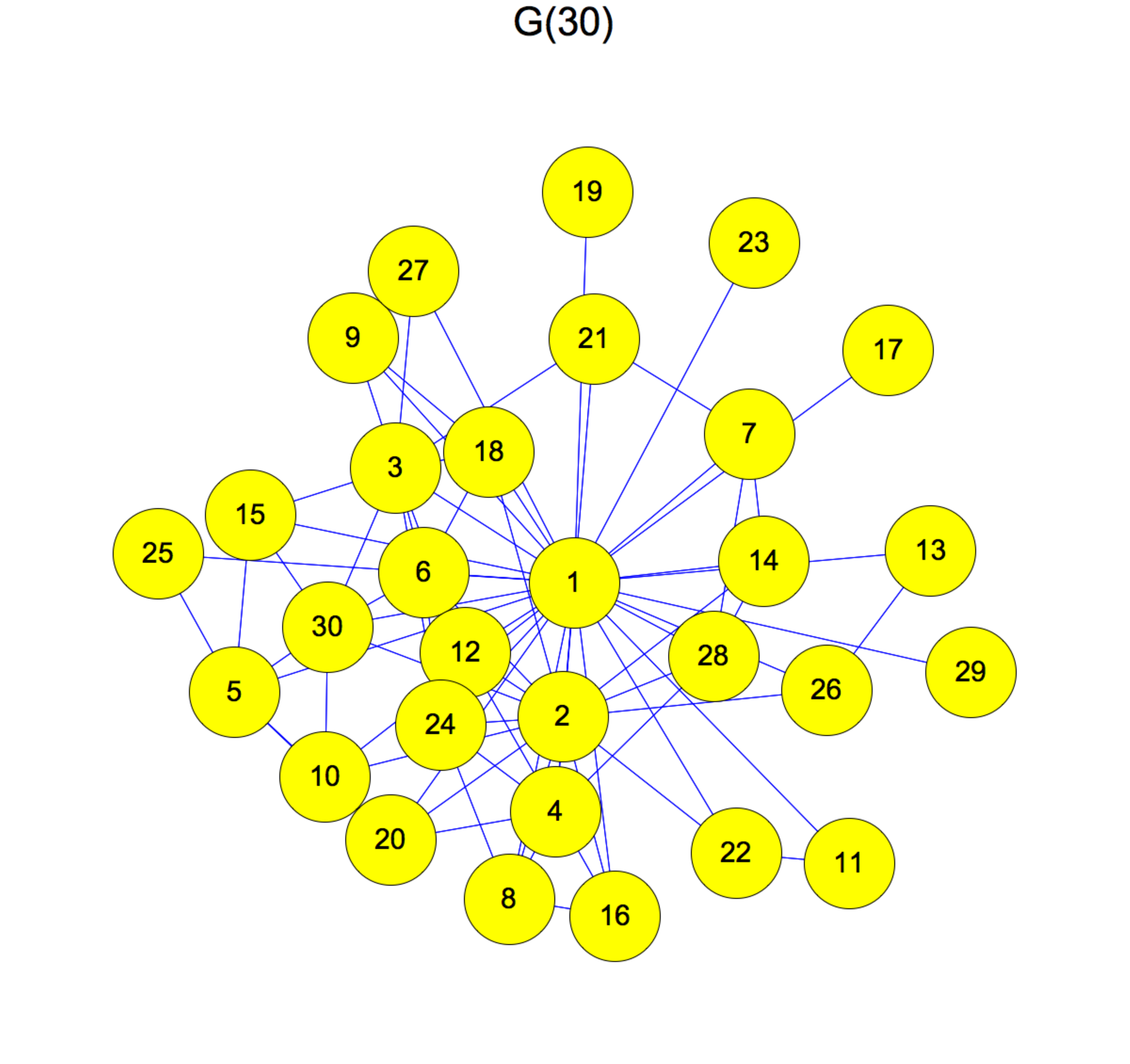}}
\scalebox{0.1}{\includegraphics{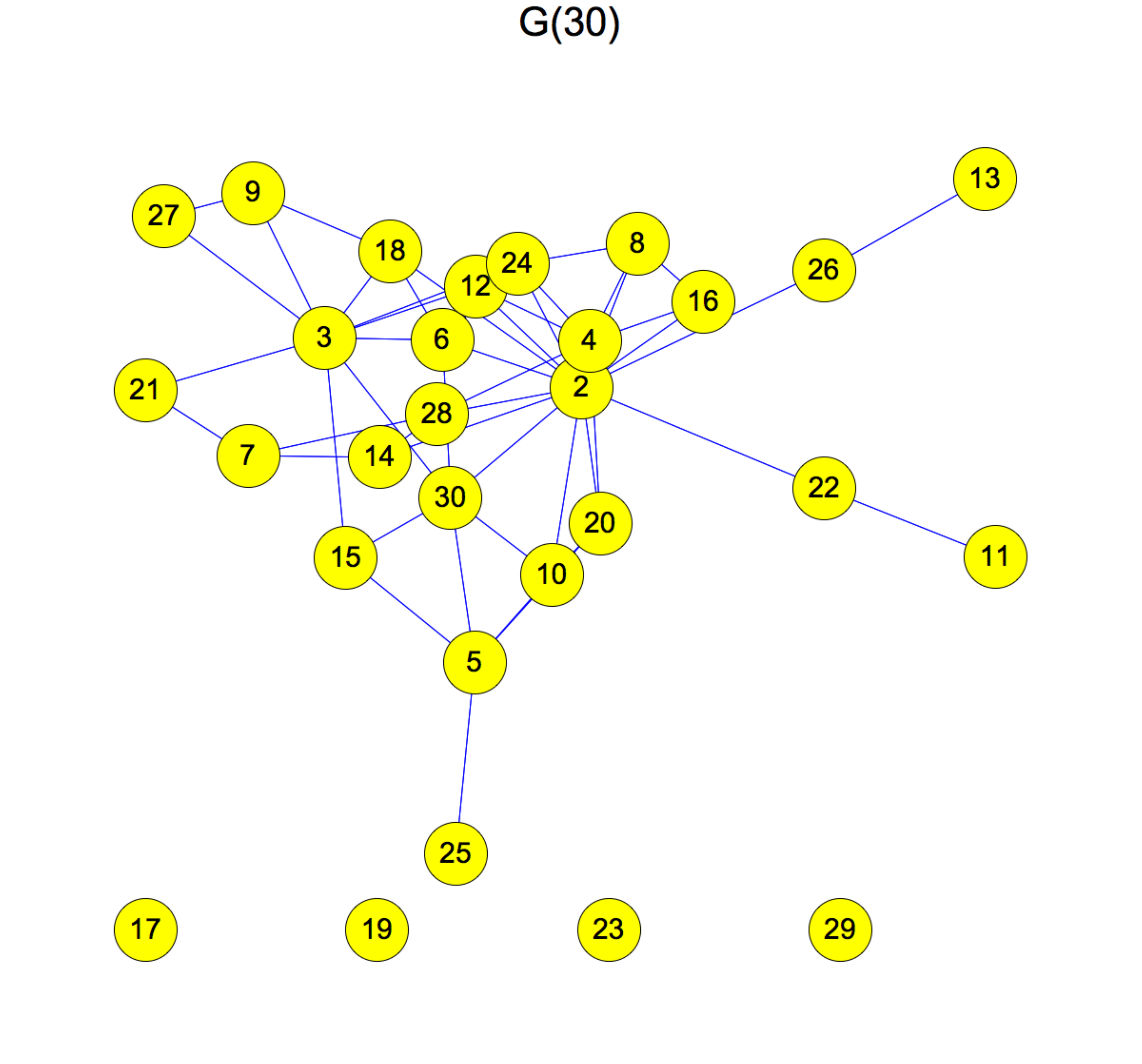}}
\scalebox{0.1}{\includegraphics{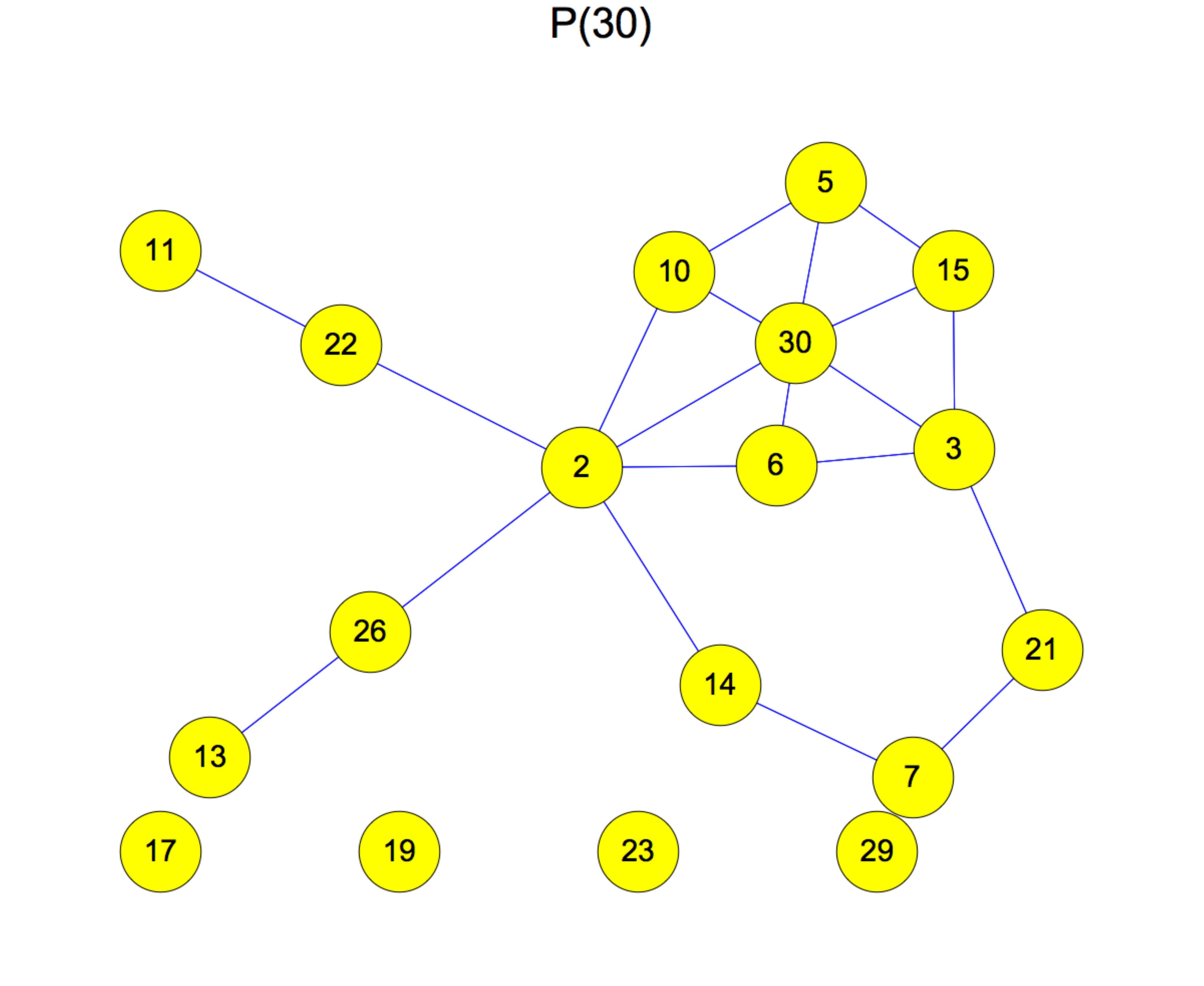}}
\caption{
The integer graph $G(30)$ contains irrelevant parts like the $3$-simplex $2,4,8,16$
and is the unit ball $B(1)$ of $1$. As a ball it always has Euler characteristic $1$. 
Removing $1$ gives the unit sphere $S(1)$ which after
removing the homotopically irrelevant parts gives the 
prime graph $P(30)$ which is a Morse cut of the Morse function
$f(x)=x$ on the Barycentric refinement $P$ of the complete graph on the spectrum of the
ring of integers $Z$. It is homotopic to $S(G(n))$ and has 
shifted cohomology and satisfies
$b_0(P(n)) = H^0(P(n)) = \pi(n)$, the prime counting function. 
The 2-disk $B(3)$ in $P(30)$ has been added when the point 
$2 \cdot 3 \cdot 5$ was joined. It decreased the Betti number $b_1$ by $1$.
The prime graphs $P(n)$ are not only simpler, they are also more naturally defined
as the Barycentric refinement of the spectrum of the commutative ring $Z$ 
rather than a semi ring $N$ like the integer graphs $G(n)$. The fact that the
cohomologies agree is a special case of the fact that Morse cohomology is equivalent
to graph cohomology. }
\label{figure1}
\end{figure}


\begin{figure}[!htpb]
\scalebox{0.14}{\includegraphics{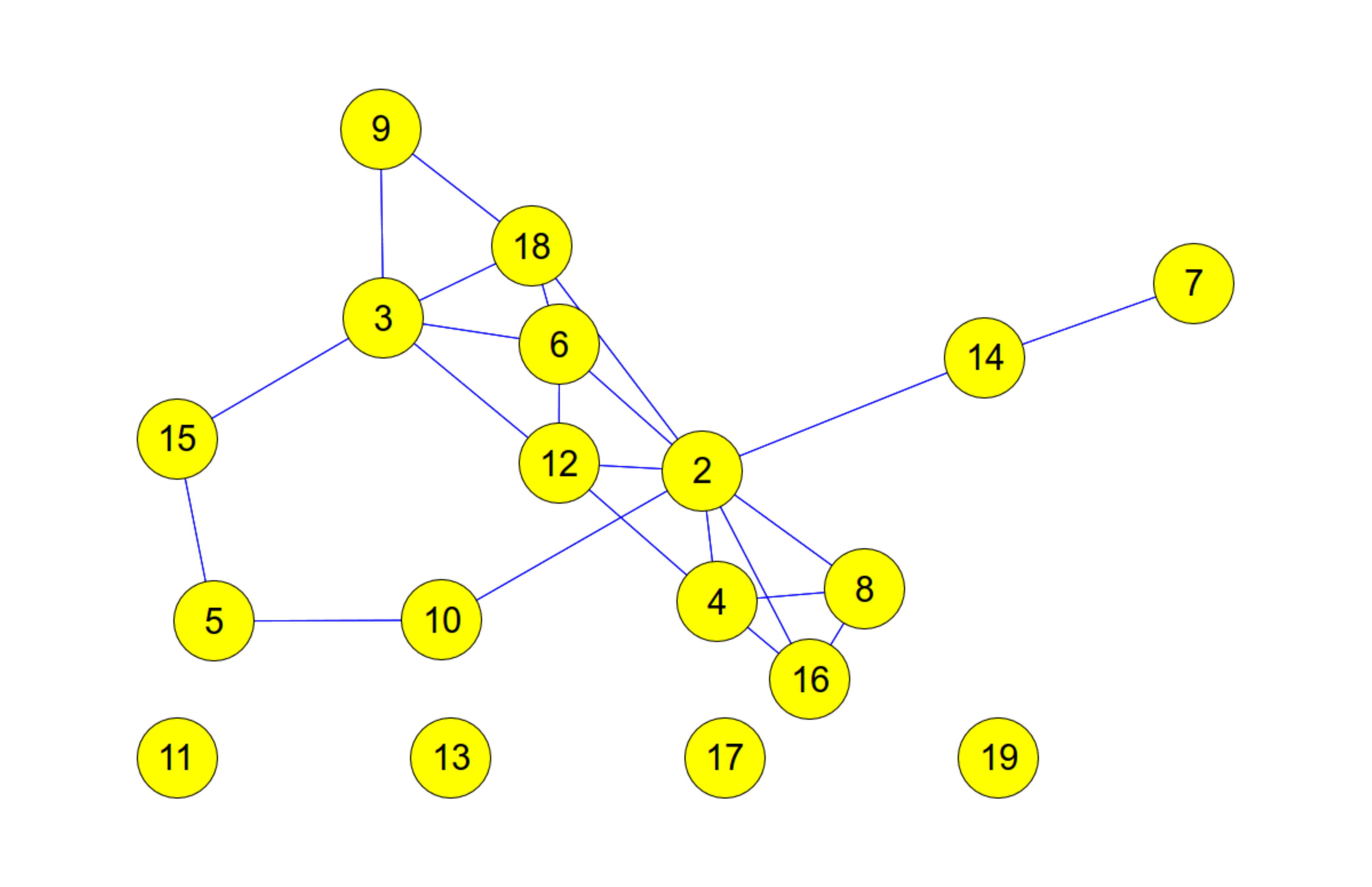}}
\scalebox{0.14}{\includegraphics{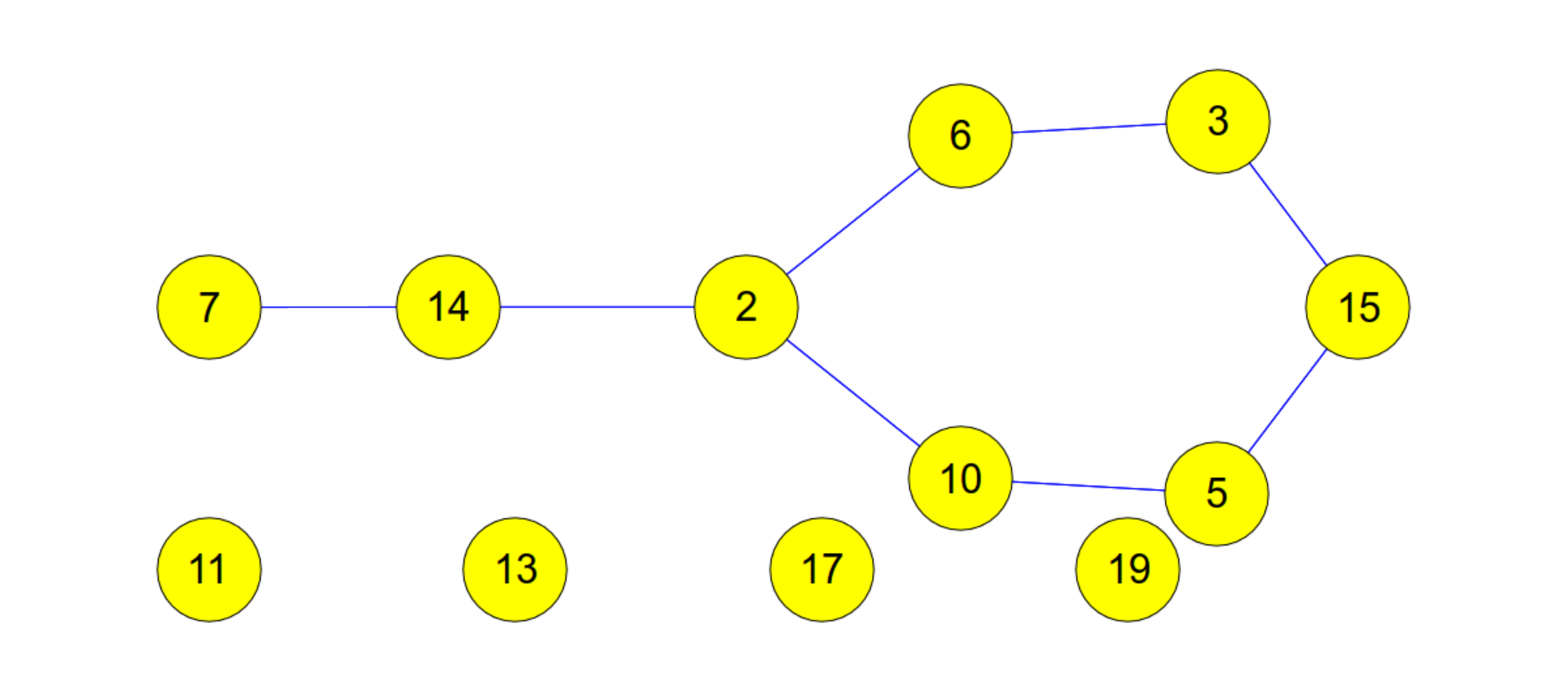}}
\caption{
In this figure we see first 
the integer graph containing all integers $2, \dots, 19$
as vertex set. As before, the unit $1$ is removed
so that we only see the unit sphere of $1$. 
The graph is homotopic to the prime graph $G(19)$ seen to the right.
which only contains the square free integers. These integers 
are critical points of the counting function $f(x)=x$ which 
has the M\"obius function $\mu(x)$ as Poincar\'e-Hopf index. 
The fact that the M\"obius function is zero on integers containing
a square of a prime reflects the fact that when adding such a vertex
nothing interesting happens topologically: adding such a vertex
is a homotopy deformation similarly as in topology where a change of
the value $c$ of a Morse function does not change the homotopy of 
$\{ f \leq c \}$. 
}
\label{figure2}
\end{figure}


\begin{figure}[!htpb]
\scalebox{0.14}{\includegraphics{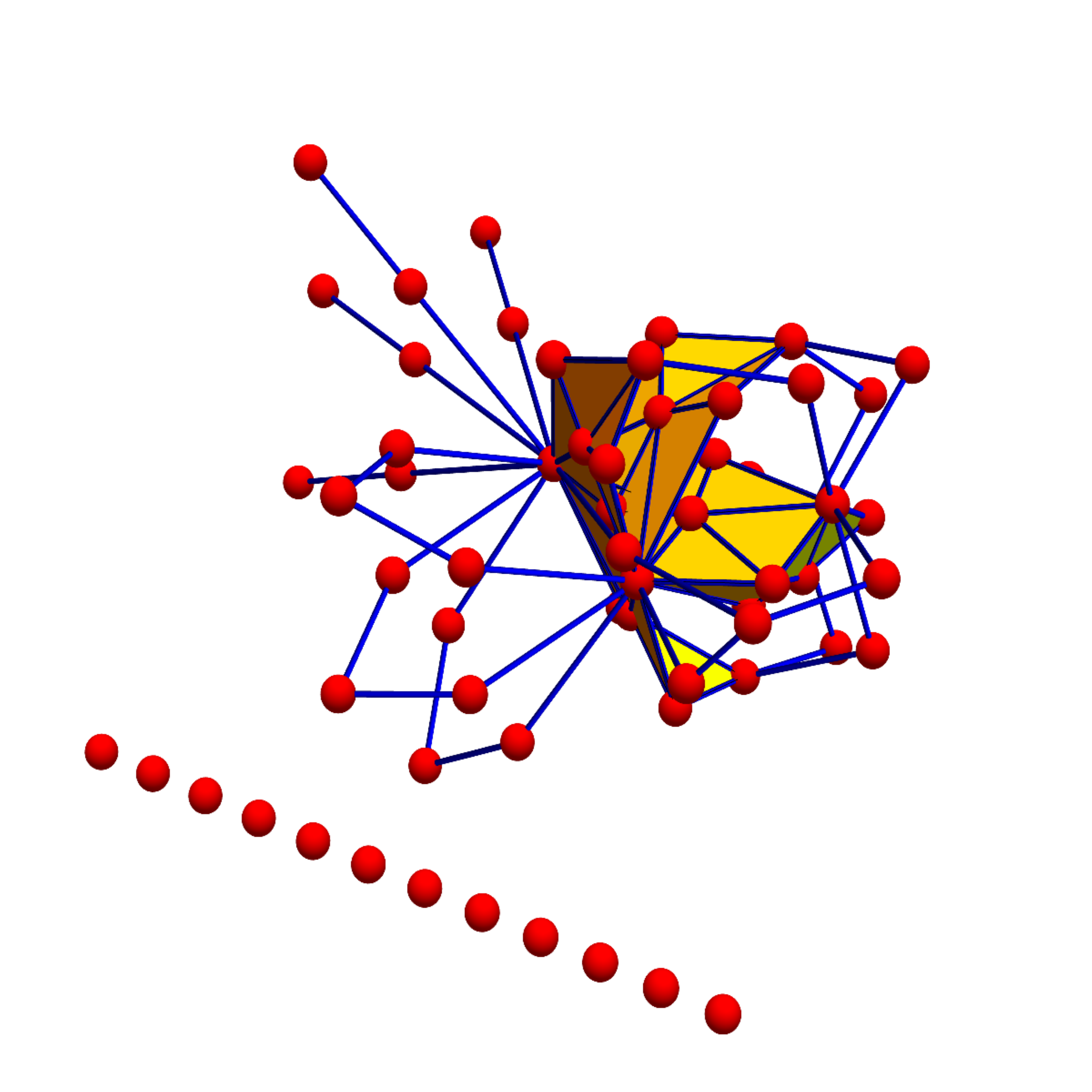}}
\scalebox{0.14}{\includegraphics{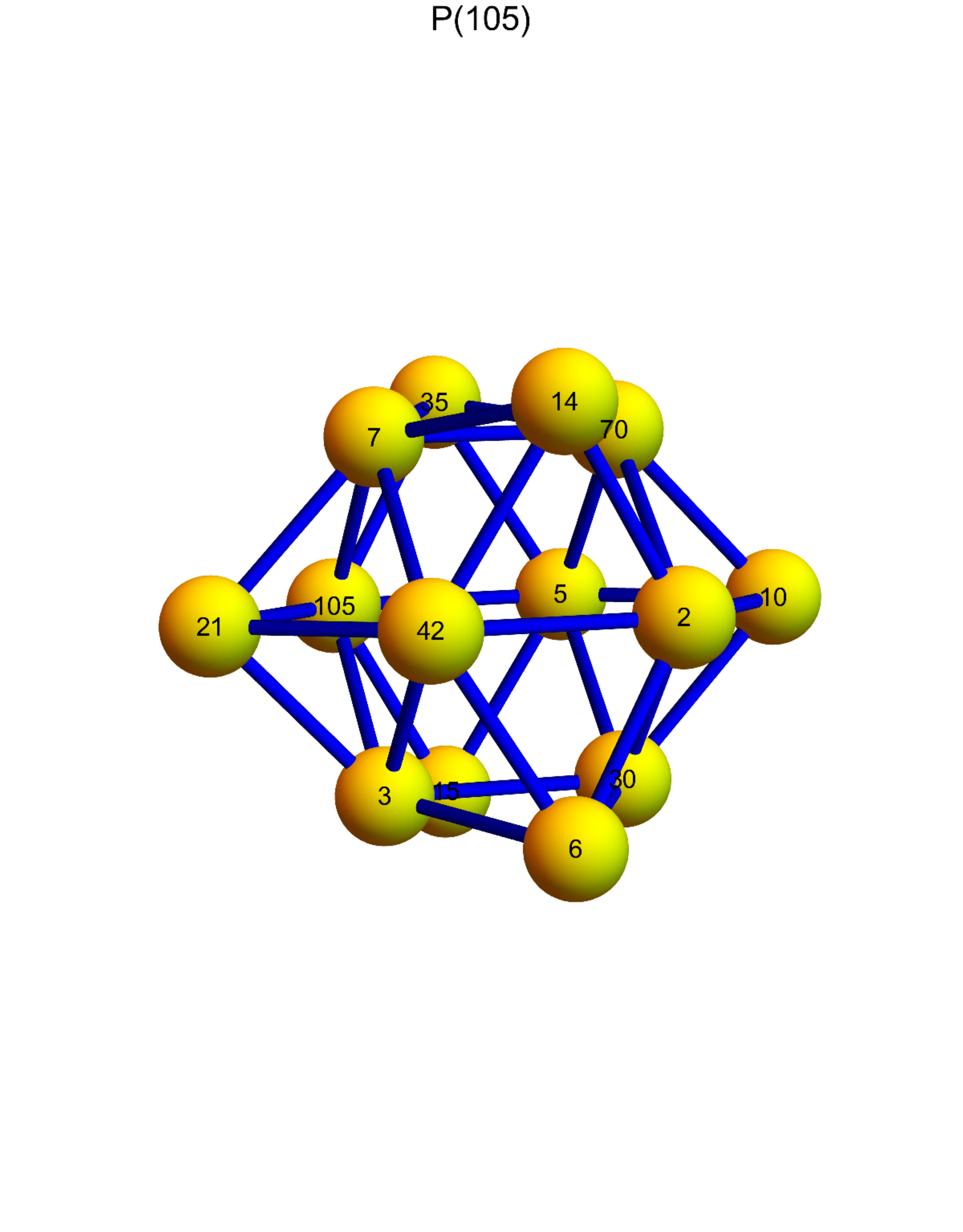}}
\caption{
The prime graph $G(105)$ is the smallest $G(n)$ containing a 
$2$-sphere contributing to a positive $b_2$. Its sphere is $S(210)$ in $G(210)$
is seen to the right and is defined by the 4 primes $2,3,5,7$. 
It is a {\bf stellated cube}. The number $105$ is 
the smallest integer for which we have a vertex with Morse index $m(x)=2$. 
This sphere stops to matter topologically again when the point $210=2 \cdot 3 \cdot 5 \cdot 7$ 
is added, which produces a central vertex, adding a 3-dimensional 
handle. For any quadruple $p_1<p_2<p_3,p_4$ of primes, the sphere
appears for $x=p_2 p_3 p_4$ and disappears at $x=p_1 p_2 p_3 p_4$. 
It illustrates how the Morse counting function $c_k(x)$  are
number theoretical. }
\end{figure}


\begin{figure}[!htpb]
\scalebox{0.14}{\includegraphics{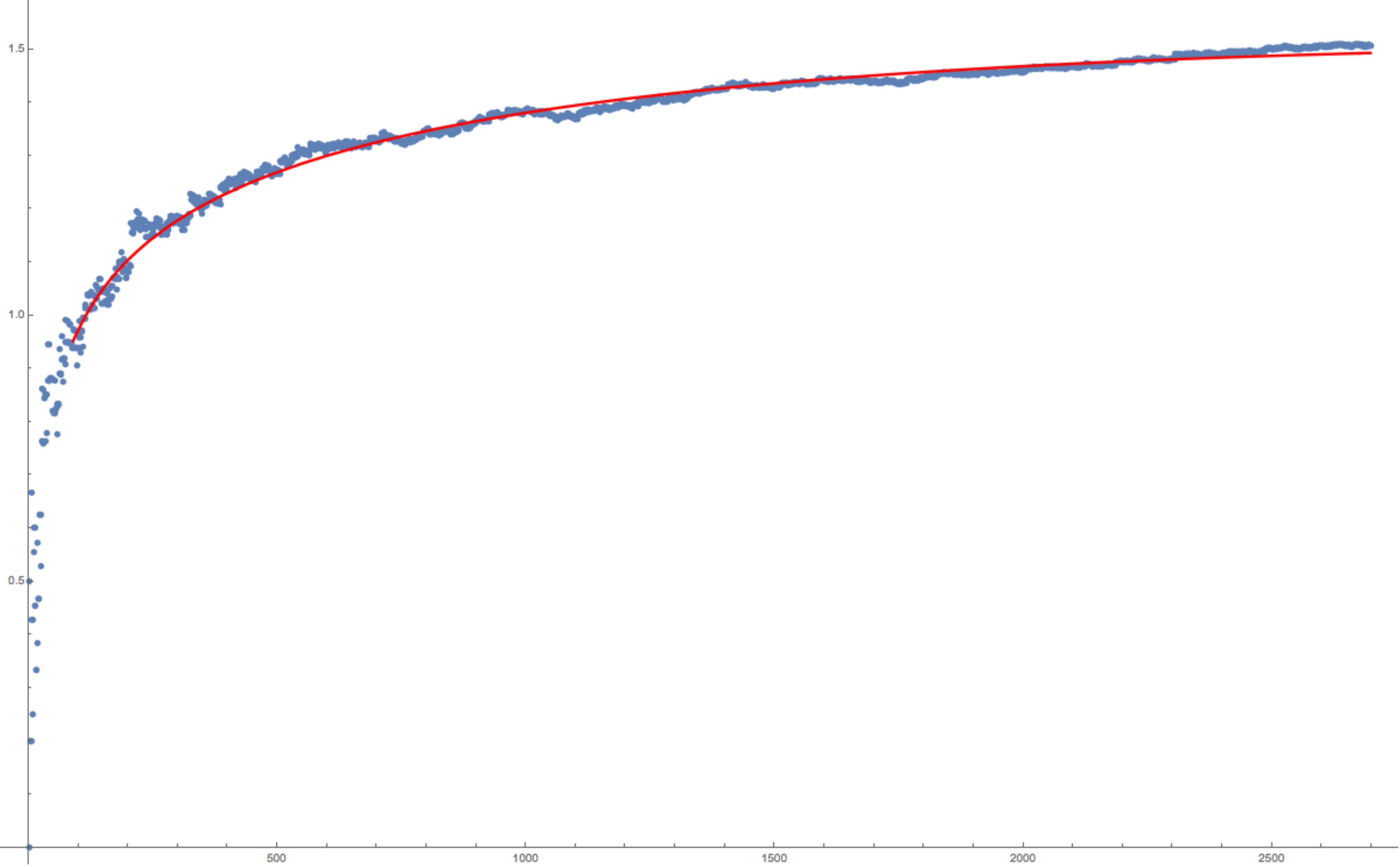}}
\caption{
The dimension of a graph is inductively defined as
${\rm dim}(G) = 1 + \sum_{x \in V} {\rm dim}(S(x))/|V|$. 
We see the graph of numerically computed dimensions of $G(n)$ 
for $6 \leq n  \leq 2690$. The best fit with the functions $1,x,\log(x)$ 
is $0.0764206 - 0.0000483082 x + 0.195795 \log(x)$. 
The growth of the dimension is related to the growth rate
of primes but the link has not yet been established.
}
\end{figure}


\begin{figure}[!htpb]
\scalebox{0.14}{\includegraphics{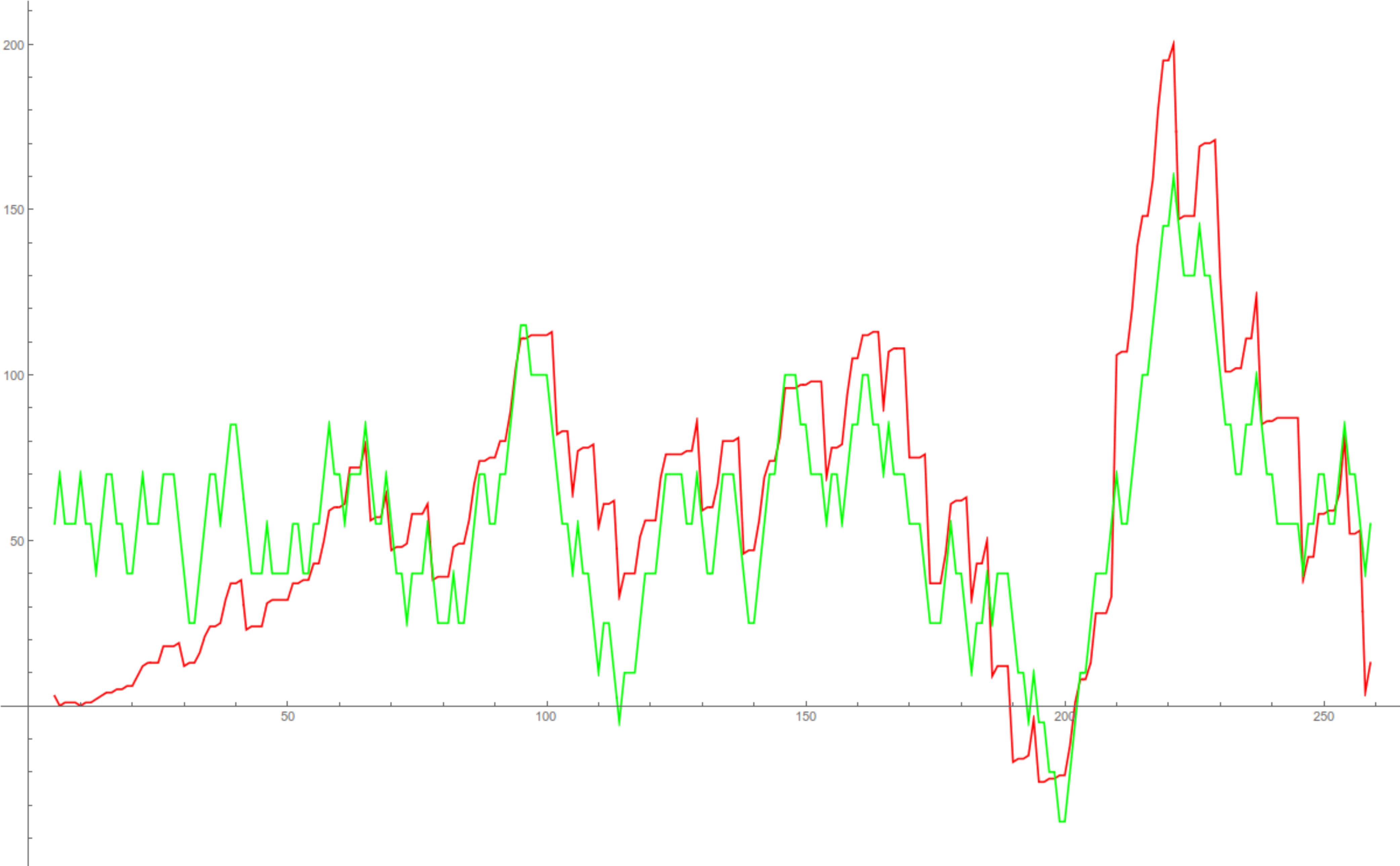}}
\caption{
The Wu characteristic $\omega(G)$ and an affinely scaled Euler characteristic 
$100-15*\chi(G)$ of the prime graphs $G(n)$ for $n \leq 259$. 
We initially constructed the prime graphs $G(n)$ as example graphs to 
investigate the Wu characteristic. 
}
\end{figure}

\section{More Remarks} 

\myparagraph{}
$b_0$ increases if $x=p$ is a prime and decreases if it is connected 
to the main connected component which happens if $x=2p$.  
$b_1$ increases if $x=p_2 p_3$ for $p_1<p_2<p_3$ and decreases if $x=p_1 p_2 p_3$. 
$b_2$ increases if $x=p_2 p_3 p_4$ for $p_1<p_2<p_3<p_4$ and decreases if 
$x=p_1 p_2 p_3 p_4$. 
The Betti number $b_3$ increases for $x=p_2 p_3 p_4 p_5$ and decreases again
if $x=p_1 p_2 p_3 p_4 p_5$. 

\myparagraph{}
We have seen that the natural numbers $N$ define a geometric object $P$
for which the asymptotic of the Mertens function provides a 
{\bf limiting Euler characteristic} of a Morse filtration. 
This limit is not a number but an asymptotic. Famously and popularized since 
more than a century, knowing the limiting growth rate of $M(n)$ would settle the 
{\bf Riemann hypothesis} RH. We see now that RH has a topological component. 
A pioneer of studying this question, Franz Mertens, was also the 
calculus teacher of Erwin Schroedinger. It was probably realized long 
before Mertens that an estimate $M(n) \leq n^{1/2+\epsilon}$ 
for all $n$ would imply $1/\zeta(s)$ to be analytic in $Re(s)>1/2+\epsilon$. 
as the Riemann functional equation implies then the roots of $\zeta$ are then
away from the strip $[1/2-\epsilon,1/2+\epsilon]$. 

\myparagraph{}
Feller \cite{Feller} illustrated the popular Mertens reformulation of the RH in a probabilistic
way: if the M\"obius function $\mu(n)$ were sufficiently {\bf random},
then by the {\bf law of iterated logarithms} the estimate 
$M(n) \leq \sqrt{n \log \log(n)}$ would hold. Already Stieljes, in 1885, showed interest 
in the Mertens function in relation to the Riemann hypothesis. 
The probabilistic intuition supports the believe that RH is reasonable 
and also explains why the initial conjecture of Mertens (which was already 
conjectured in a letter from Stieljes to Hermite) that $M(n) \leq \sqrt{n}$ is
too strong. Indeed, Odlyzko and te Riele \cite{OdlyzkoRiele} disproved it thirty years ago. While it is
still not known whether $M(n)/\sqrt{n}$ is bounded, the probabilistic heuristic
of Feller of the iterated log theorem makes it likely that $M(n)/\sqrt{n}$ indeed 
will grow like $\log(\log(n))$. According to \cite{OdlyzkoRiele} also a conjecture by Good and Churchhouse 
asking whether $\limsup M(n) (x \log(\log(x)))^{-1/2} = \sqrt{12}/\pi$ appears too optimistic. 

\myparagraph{}
The free Riemann gas or {\bf primon gas} model \cite{Spector} gives to a vertex $p$ in the
set of primes the energy $\log(p)$. Since the energy of a simplex $n=p_1\dots p_k$
is $E(n) = \log(n)$, it defines an energy on the Barycentric refinement, which in
this model leads to the {\bf second quantized Hamiltonian}. It is the Pauli inclusion principle
which leads to the graph $G$. The function $\mu(n) = (-1)^k$ is the Fermion operator.
It is positive for bosons, negative for fermions. While
without exclusion principle, the {\bf partition function} of
the model at {\bf inverse temperature} $s=1/T$ is
$\sum_n e^{-E(n)/T} = \sum_n e^{-\log(n) s} =  \sum_n 1/n^s=\zeta(s)$,
the partition function with Fermion statistics is the Dirichlet character
$\sum_n \mu(n) e^{-\log(n) s} = \sum_n \mu(n)/n^s = 1/\zeta(s)$.
In this model, the Riemann hypothesis is a statement about the partition function.
The pole at $s=1$ is the {\bf Hagedorn temperature}.

\myparagraph{}
In the continuum, there are estimates of Gromov 
\cite{Gromov81} of the type $b_k \leq C^{1+\kappa D}$ 
where $D$ is the diameter $-\kappa$ the minimal sectional curvature of the manifold.
For graphs $G(n)$ we indeed have an estimate $\kappa D \leq \log(n)/\log\log(n)$ which
is better than the Gromov estimate but still not strong enough to prove RH. 
Constants are important and as expected, topology alone is hardly the
silver bullet. Still, topology motivates to look at discrete 
differential geometric notions like {\bf curvature} and its connection to Betti numbers
in general. Indeed it provokes the question whether in full generality, for any finite
simple graph, a Gromov type estimate $b_k(G) \leq e^{c_k \kappa(G) d(G)}$
holds for finite simple graphs, where kappa is the minimal sectional curvature 
and $d(G)$ the diameter of the graph. 

\myparagraph{}
We first we have to define the {\bf minimal sectional curvature} $\kappa$ for a finite simple graph. 
Let $\kappa$ be the maximum for which the sectional curvature 
satisfies $1-{\rm deg}_W(x)/6 \geq -\kappa$ for all vertices $x$ and {\bf wheel sub graphs} $W$. 
For $G(n)$, the curvature is bounded by $1-\nu(n)/3$, where $\nu(n)$ is the maximal 
number of distinct prime factors of $n$. We know $\nu_n \leq \log(n)/\log\log(n)$.

\myparagraph{}
The {\bf diameter} of the graph $G(n)$ is bounded above by $5$. Proof: take two vertices
$k,m$. If both are not prime they both have a smaller factor $p,q$. The path
$k,p,2q,2,2q,q,m$ shows that $d(k,m) \leq 6$ because all numbers $p,2q,2,2q,q$ are 
smaller or equal than $n$ so that this path is in $G(n)$. The distance is
even $\leq 4$ then as if $p$ is the smallest prime factor of $k$ and $q$ the smallest of $m$, then 
$pq \leq \sqrt{k m} \leq max(m,n)$ so that $k,p,pq,q,m$ is in $G(n)$ and $d(k,m) \leq 4$.
If $k,m$ have a common prime factor $p$, then their distance is even 
$2$ as the path $k,p,m$ shows. 
If $k=p$ is prime but connected to the main component, then there is a point 
$p l \leq n$ so that $2p \leq n$. The path $p,2p,2$ is in $G(n)$ 
meaning $d(p,2) \leq 2$. We see that for any integer $m$, $d(m,2) \leq 3$. This shows that the
distance in that case is $\leq 5$ with connection $p,2p,2,2q,q,m$ of length $m$.
If both $n,m$ are prime their distance is $\leq 4$.
All graphs $G(n)$ with $n \geq 15$ have diameter $5$ where the maximal
distance is obtained between a prime $p$ and a composite number $m=q r$ with
shortest connection $p,2p,2,2q,q,m$ which can not be shortened.
For the graph $G(15)$ it is the distance between $11$ and $15$ as $11,22,2,10,5,15$ 
and $11,22,2,6,3,15$ are the two shortest paths from $11$ to $15$. 

\myparagraph{}
In our case, the Gromov formula gives a bound on the Betti numbers 
$C^{ 5 \log(n)/\log \log(n) }$ which is only polynomial in $n$. 
Using the simplest Morse inequality $b_k(G(n)) \leq c_k(G(n))$ 
we can see better as $c_k$ grows like the number $\pi_k(n)$ telling how
many prime $k+1$-tuples $p_0<p_1<p_2< \dots < p_k \leq n$ there are. 
Neglecting constants we have $c_k \leq n^k/( \log(n)^k k!)$. 
The Morse inequalities now gives us information about the alternating
sum $\sum_k (-1)^k c_k$ in terms of the growth of the Mertens
function. We see this here reflected a relation 
between the prime number distribution and the growth of the Mertens function
related to the Euler {\bf golden key} 
$\sum_n \mu(n)/n^s = 1/\zeta(s) = \prod_k ( 1 - 1/p_k^s)$.


\myparagraph{}
We initially have used the graphs $G(n)$ as a play ground to investigate the 
{\bf Wu characteristic} $\sum_{x \sim y} (-1)^{ {\rm dim}(x) + {\rm dim}(y)}$
summing over all pairs of interacting simplices $x,y$ (see \cite{valuation}). 
This number is like Euler characteristic a functional on graphs which is 
multiplicative and satisfies Gauss-Bonnet or Poincar\'e-Hopf formulas.
It is natural to ask whether the growth rate of $\omega(G(n))$ has any 
relation to analytic properties of some zeta function.

\myparagraph{}
Geometry enters basic counting principles.
Geometric arrangements of pebbles on rectangular patterns motivated 
on a fundamental level to accept commutativity laws 
$n \times m = m \times n$ despite that on a fundamental level, nature does
not multiply commutatively as the Heisenberg anti-commutation relation $pq-qp=\hbar$ shows.
Neglecting square parts is not only important in combinatorics
it is also the reason for fundamental principles like {\bf Pauli exclusion}.
In our geometric picture, adding numbers with square factors is a homotopy which does not
change Euler characteristic. It is the Fermionic nature of calculus which fundamentally 
evaluates on oriented substructures and changes sign when the orientation is
switched. With a number like 12, switching two prime factors 2 would change the
sign but not the number. It is the interpretation of numbers as simplices in a 
simplicial complex which makes this clear. 

\myparagraph{}
Our understanding of what a "point" is has changed over time: 
while the first mathematicians saw points as marks or pebbles without formal
definition, it was Euclid who first saw the need to define it. He decided to 
declare a points to be something "that which has no part".
Descartes started to access points through algebra. This has been
evolved more and more since. Points are now seen as maximal 
ideals in a commutative ring. In analysis, the Gelfand representation theorem 
identifies points of a $C^*$ algebra as elements in the spectrum. 
The Nullstellensatz identifies points of an affine variety 
as maximal ideals of the ring its regular functions.
While already Kummer and Noether saw prime ideals as generalized points,
irreducible varieties like an elliptic curve $y^2=x^2+x+1$ are now seen
as "points" in a "spectrum" on which Zariski built a topology in which 
closed sets are the set of prime ideals which contain a fixed prime ideal.
Grothendieck took the picture that prime ideals are points more seriously
even so he considered the notion of a "scheme" as a setup of 
"infantile simplicity". But it is the simplicity which makes it appealing.
The ring of integers as the simplest scheme 
has been studied since the very beginning of mathematics.

\myparagraph{}
An partial order structure given by an abstract simplicial complex can not 
be visualized well. Fortunately, after a Barycentric subdivision, any 
abstract finite simplicial complex is the Whitney complex of a finite simple
graph. Graphs are more intuitive as they can be drawn in such a way that the
vertices and edges alone determine all the topological 
information of the original simplicial complex. Its cohomology or Euler characteristic 
for example is the same. The graph category also comes naturally with analytic 
structures like incidence matrices and Laplacians. However, restricting to 
graphs is merely language as the Barycentric refinement of any simplicial
complex is already the Whitney complex of a graph. It is useful language
however as it is a basic data structure known to many computer languages. 

\myparagraph{}
For a ring in which the unique factorization fails, the Morse condition 
can fail. A case, where things still work is the ring of Gaussian integers. 
There is a canonical counting function as we can only primes in equivalence classes. 
This means, we look at primes in the quotient $\mathbb{C}/D_4$,
where $D_4$ is the dihedral group generated by the units in $\mathbb{C}$ and conjugation.
A basic question is for which factorization domains we get a Morse picture? 
For fields, the graph is empty. The simplest examples with non-trivial graph 
are the rings $\mathbb{Z}/\mathbb{Z}_n$ with non-prime $n$. 

\myparagraph{}
Morse theory is related but stronger than Poincar\'e-Hopf. One can see
this well when viewing the Poincar\'e-Hopf formula as a special
case of the Morse inequalities. In fixed point
theorems, Morse theory gives better bounds and more sophisticated
theories like Floer cohomology have demonstrated this.
While one get to the sum of Betti numbers from the weak Morse inequalities,
The Lefschetz theorems for an automorphism $T$ only gives the 
alternating sums of the Betti numbers. 
By the way, in the discrete the Lefschetz fixed point theorem can be proven quickly from the heat flow:
the {\bf Lefschetz number} is the super trace $\chi_T(G)$ of the induced map $U_T$ on $H^p(G)$.
The theorem tells that it is equal to $\sum_{T(x)=x} i_T(x)$, where $i_T(x)=(-1)^{{\rm dim}(x)} {\rm sign}(T|x)$
is the {\bf Brouwer index}. For $T=Id$, the Lefschetz formula is {\bf Euler-Poincar\'e}.
With the {\bf Dirac operator} $D=d+d^*$ and {\bf Laplacian} $L=D^2$, discrete {\bf Hodge} tells that
$b_p(G)$ is the nullity of $L$ restricted to $p$-forms.
By {\bf McKean Singer super symmetry}, the positive Laplace spectrum on even-forms is the
positive Laplace spectrum
on odd-forms. The super trace ${\rm str}(L^k)$ is therefore zero for $k>0$ and $l(t)={\rm str}(\exp(-tL) U_T)$
with {\bf Koopman operator} $U_Tf=f(T)$ is $t$-invariant. Lefschetz follows because
$l(0) = {\rm str}(U_T)$ is $\sum_{T(x)=x} i_T(x)$ and  $\lim_{t \to \infty} l(t)=\chi_T(G)$  by Hodge.

\myparagraph{}
We have seen a more poetic rather than useful story
about counting. It relates two seemingly unrelated topics: the
{\bf arithmetic of integers} with its enigma of the structure of primes
as well as {\bf Morse theory} in {\bf combinatorial topology} which has
especially in conjuncture of cohomology theories and semi-classical analysis
tools become a heavy machinery.
Morse theory initially was built to study topological spaces and been
developed in a differential topological frame works.
There are various approaches possible to discrete Morse theory. One theory
was built by Forman \cite{Forman1999} and the description here is motivated by that.
Especially important is the idea to look at simplices as "generalized points".
The idea to extend objects has been frequent in mathematics like extending
number systems, the use of generalized functions or the shift from maximal ideals to
prime ideals. In the context of simplicial complexes it is very natural and
Barycentric refinement shows that the original simplices are then vertices.

\myparagraph{}

Here are some mathematical results.
Let $G(n)$ have the square free integers $\geq 2$ as vertices.
Two integers are connected if one divides the other. 

\begin{propo}
The diameter of the connected component of $P(n)$ is $\leq 5$ 
for all $n$. 
\end{propo}
\begin{proof} 
The distance of any composite number $k=p q$ to $2$ is $\leq 3$ as 
the path $k,p,2p,2$ shows.
If $p$ is a prime connected to the main component, then 
there exists $q$ with $pq \leq n$ and so $2p \leq n$ and the path $p,2p,2$
shows $d(p,2)\leq 2$. If two integers $k,m$ are both
composite with $k=pq, m = rs$, then pick the smaller factors $p,r$ of
both and the distance between $k,m$ is $\leq 4$ with path  $k,p,pr,r,m$. 
If $k,m$ are both prime connected to the main component then their
distance is $\leq 4$ as $k,2k,2,2m,m$ shows. 
\end{proof}

This should be true for any $R$ be an integral domain with a 
non-trivial absolute value function $| \cdot |$. 
The next lemma tells that $f$ is a Morse function. What actually happens is that 
every sphere $S_f^-(x)$ is of a Hopf fibration type in the sense
they are the union of complementary solid tori glued  along a torus. 
Here is the key result which proves that $1-\chi(S^-(x)))$ is always either $-1$ or $1$,
depending on whether $\mu(x)=-1$ or $1$. 

\begin{propo}
For every vertex $x$ in $G$, the graph $S=S^-_f(x)$ is an Evako 
$m$-sphere for some $m \geq -1$.  Consequently, $1-\chi(S^-(x))$ is $\mu(x)$. 
\end{propo}
\begin{proof}
If $x$ is prime, then $S^-_f(x)$ is the empty graph, the $(-1)$-sphere.
If $x=pq$ is a product of two primes, then $S^-_f(x) = (\{p,q\},\{\})$ is the $0$-sphere. 
If $x=pqr$ is a product of three primes, then $S^-f(x)$ is $C_6$ with vertices 
$p,pq,q,qr,r,rp$. If $x=pqrs$ is a product of four primes, then $S^-_f(x)$ 
is a stellated cube on which every unit sphere is either a $C_6$ graph (for the products
of three primes or primes) or then a $C_4$ graph for the products of 2 primes like
$pq$ for which the unit sphere is $p,pqr,q,pqs$. In general, the sphere is a 
\end{proof}

Also this should be true if the ring $Z$ is replaced by any unique factorization domain. 
For a general finite simple graph $G$ and a scalar function on the vertex set, we 
call $f$ a {\bf Morse function} if this property is satisfied. 
This result is related to the fact that if $G$ is an arbitrary finite simple graph and 
$f$ an arbitrary injective function on its vertex set producing an ordering and if $G_1$ 
is the Barycentric refinement $f_1$ is the function which assigns to a new vertex the place in which
a monoid appears in the Stanley-Reisner ring representation of the graph which produces
a function on the vertex set of $G_1$, then $i_f(x)$ is always either $-1$ or $1$ 
depending on the dimension of the sphere $S_g^-(x)$.

\begin{propo}
For any Morse function $f$ on a finite graph, the weak 
Morse inequalities $b_k \leq c_k$ hold.
\end{propo}
\begin{proof}
Start with the minimum. At every critical point  of degree $m$, the $m$'th
Betti number changes. This gives $b_m \leq c_m$. As $c_m$ grows monotonically
and $b_m$ can go up and down we have soon a strict inequality. 
\end{proof}

This shows that $p(t)-c(t)$ has positive coefficients.
The Morse-Hopf relation $p(-1)=c(-1)$ is a special case of the
Poincar\'e-Hopf theorem. One can also use super symmetry \cite{Cycon} to show that the 
supersum of $p-c$ is zero: $\sum_{odd} c_p-b_p = \sum_{even} c_p-b_p$. 

\begin{propo}
For any Morse function $f$ on a finite graph the strong inequalities
hold: $p(t)-c(t) = (1+t) r(t)$ where $r(t)$ has positive coefficients. 
\end{propo}
\begin{proof} 
This is equivalent to $(-1)^p \sum_{k=0}^p (-1)^k (c_k-b_k)  \geq 0$. 
It starts with $c_0-b_0 \geq 0, (c_0-b_0) - (c_1 - b_1) \leq 0$. 
For even $p$, it means that $str( c-b) \geq 0$ and for odd $p$
that $str(c-b) \leq 0$. In the discrete, there is a simple proof
as we can change the simplicial complex structure on the graph to be
the $p-1$-sceleton complex. 
\end{proof}

Classically, semi classical analysis separates away a group of eigenvalues for
which super symmetry holds. The intuition is that the critical points "under water"
which have appeared and disappeared produce "wells" in which some quantum particles
are trapped. These modes are created with a deformation of the Laplacian called
Witten deformation $d_s=e^{-s f} d e^{sf}$ defines a family of Laplacian $L_s$
for which the kernel does not change.  
Super symmetry respects this gap in the sense that the deformed Dirac operator $D_t$ maps 
for every $p$ the low valued eigenvalues onto each other: \\

We have seen that a Barycentric refinement $G_1$ of a graph $G$ 
has many Morse functions and especially the function $f(x) = {\rm dim}(x)$, 
where dim is the dimension of the simplex $x$ in the original graph. 
Lets call the pair $(G_1,f)$ of a finite simple graph $G$ a {\bf Barycentric Morse system}.
Its {\bf exterior derivative} is defined as
$$  dg(x) = \sum_{y} n(x,y) g(y) $$ 
with $n(x,y)$ being $1$ if the sphere $S^-_f(y)$ has the same orientation than the 
induced orientation from $S^-_f(x)$ and $-1$ else. The orientation of the sphere $S^-_f(x)$
is the orientation of the corresponding simplex $x$. Note that $S^-_f(x)$ is just
the simplicial complex of $x$ where the set $\{x_1, \dots, x_m\}$ itself is excluded. 

\begin{propo}
For the Barycentric Morse system $G_1,f={\rm dim}$, the Morse cohomology 
is simplicial cohomology.
\end{propo}
\begin{proof}
Since the number of critical points of Morse index $m$ in $G_1$ 
is the number of $m$-simplices in the graph $G$ and the exterior 
derivatives correspond, the simplicial exterior derivative being
$dg(x_0, \dots, x_{m-1}) = \sum_{k=0}^{m-1} (-1)^k g(x_0, \dots \hat{x}_k, \dots, x_{m-1})$, 
and the Morse derivative being
$dg(x) = \sum_{y} n(x,y) g(y)$,
the two complexes are naturally isomorphic.
\end{proof}

Also this is a result for arbitrary finite simple graphs. 
Especially, the enumeration function on Barycentric refinement 
defines a Morse-Smale system. \\

Finally, lets see why graph theory is good enough for describing
simplicial complexes. We don't look at geometric simplicial complexes which 
are defined in Euclidean space but look at them purely combinatorially. 
An {\bf abstract simplicial complex} on a countable set V
is a set $K$ of finite subsets of $V$ such that if $A$ is in $K$ and $B$ is a 
subset of $A$, then $B$ is in $K$. The {\bf Barycentric refinement} of
$K$ is defined on the countable set $V_1=K$. A set of elements in $V_1$
is in $K_1$ if it is the power set of some set $A$ in $K$.
Given a graph $G=(V,E)$, it defines an abstract simplicial complex $K$ on $V$
called the {\bf Whitney complex}. The elements of $K$ are the complete 
subgraphs of $G$. 

\begin{propo}
Given any abstract simplicial complex $K$ on a set $V$, then 
its Barycentric refinement is the Whitney complex of a graph.
\end{propo}
\begin{proof}
The vertex set is $K$, two vertices are connected, if one is contained
in the other. This gives a graph $G$. The complete subgraphs of $G$ 
are the elements of $K_1$. 
\end{proof}

For classical Morse cohomology, see \cite{SchwarzMorse}. 
In order to develop Morse cohomology in more generality, we need a notion of 
stable and unstable manifold. We assume here that $G$ is a $d$-graph, meaning that
every unit sphere $S(v)$ is a $(d-1)$-sphere. A first result is the existence of 
stable manifolds for a Morse function $f$. Assume $G_1$ is the Barycentric 
refinement of $G$ and that $f$ is a Morse function on $G_1$ meaning that it is 
locally injective and that every unit sphere $S(x)$ is a $m-1$-sphere. Let $y$ be
a point on $S(x)$. Then the unit sphere $A(y)$ of $y$ in $S(x)$ is a $m-2$ sphere. The
connection from $x$ to $y$ has a unique geodesic extension defining a new point $z$ in $S(x)$
such that $x,y,z$ is a geodesic piece. Only add this point if $f(z)<f(y)$. 
The suspension of $A(y)$ is the graph generated by the vertices of $A(y)$ and $\{x,z\}$. It
is an other $m-1$ sphere. Continue extending the graph until no continuation is possible
any more. The {\bf unstable manifold} is the stable manifold of $-f$. 
We say a function $f$ is Morse-Smale, if for every critical point $x$ of Morse index $m$
and every critical point $y$ of Morse index $m+1$, either $W^-(x)$ and $W^+(y)$ do not intersect
or that $W^-(x) \cap W^+(y)$ is transversal in the sense that $W^-(x) \cap (S(y))$ is 
contained in $W^-(y)$. \\

The example, where $G_1$ is the Barycentric refinement of an arbitrary finite simple graph
$G$, the function $f(x) = {\rm dim}(x)$  is Morse-Smale and the corresponding Morse cohomology 
naturally equivalent to the simplicial cohomology of the original graph $G$. We have not yet
explored more general examples but at the moment expect the theory to require some 
regularity on $G$ like that $G$ is a $d$-graph. The reason is the existence of stable and
unstable manifolds. 

\bibliographystyle{plain}

\end{document}